\documentclass[11pt]{article}
\usepackage{amsmath}
\usepackage{amssymb}
\usepackage{amsthm}
\usepackage{bookmark}
\usepackage{geometry}
\usepackage{hyperref}
\newtheorem{theorem}{Theorem}
\newtheorem{corollary}[theorem]{Corollary}
\newtheorem{lemma}[theorem]{Lemma}
\newtheorem{proposition}[theorem]{Proposition}
\theoremstyle{definition}
\newtheorem{remark}[theorem]{Remark}
\DeclareMathOperator{\N}{N}
\DeclareMathOperator{\Tr}{Tr}

\title{On a Class of Permutation Rational Functions Involving Trace Maps}
\author{Ruikai Chen\textsuperscript{1,2}\and Sihem Mesnager\textsuperscript{1,2,3}}
\date{\small\textsuperscript{1}Department of Mathematics, University of Paris VIII, F-93526 Saint-Denis\\\textsuperscript{2}Laboratory Analysis, Geometry and Applications, LAGA, University Sorbonne Paris Nord, CNRS, UMR 7539, F-93430, Villetaneuse, France\\\textsuperscript{3}Telecom Paris, Polytechnic institute of Paris, 91120 Palaiseau, France\\Emails: \href{mailto:chen.rk@outlook.com}{chen.rk@outlook.com}\quad\href{mailto:smesnager@univ-paris8.fr}{smesnager@univ-paris8.fr}}

\begin{document}

\maketitle

\noindent\textbf{Abstract.} Permutation rational functions over finite fields have attracted high interest in recent years. However, only a few of them have been exhibited. This article studies a class of permutation rational functions constructed using trace maps on extensions of finite fields, especially for the cases of quadratic and cubic extensions. Our achievements are obtained by investigating absolute irreducibility of some polynomials in two indeterminates.\\
\textbf{Keywords.} Finite Field, Extension Field, Polynomial, Permutation, Rational Function.\\
\textbf{Mathematics Subject Classification.} 11T06, 14H05.

\section{Introduction}

Let $q$ be a prime power and $\mathbb F_q$ denote the finite field of order $q$. A permutation polynomial over $\mathbb F_q$ is a polynomial that acts as a permutation on $\mathbb F_q$ in the usual way (see \cite[Chapter 7]{lidl1997} or \cite[Chapter II.8]{mullen2013}). Permutation polynomials over finite fields are exciting objects in number theory, studied widely in the last decades and employed extensively in coding theory, cryptography, combinatorics, etc. Finding permutation polynomials of simple forms and investigating their properties is still appealing and challenging.

A permutation polynomial can appear as a rational function, if the rational function defines a map from a finite field to itself. Consider some extension $\mathbb F_{q^n}$ ($n>1$) of $\mathbb F_q$ and let $\Tr$ be the trace map from $\mathbb F_{q^n}$ to $\mathbb F_q$ given by
\[\Tr(x)=\sum_{i=0}^{n-1}x^{q^i}.\]
Permutation polynomials of the form
\[L(x)+(x^q-x+b)^s\]
for some $q$-linear polynomial $L$ over $\mathbb F_{q^n}$ (a polynomial that induces a linear endomorphism of $\mathbb F_{q^n}$ over $\mathbb F_q$), an integer $s$ and $b\in\mathbb F_{q^n}$ have drawn a lot of attention in recent years. See \cite{helleseth2003new, zha2012two, li2013further, yuan2014further} for example. In particular, some necessary and sufficient conditions are given in \cite{hou2020type} and \cite{bartoli2021conjecture} for
\[x+(x^q-x+b)^{-1},\]
with $\Tr(b)\ne0$ to be a permutation rational function of $\mathbb F_{q^n}$; for the special case that $q$ is odd and $n=2$, it has been shown that it is a permutation rational function of $\mathbb F_{q^2}$ if and only if $\Tr(b)^2=1$. Furthermore, we have a related polynomial 
\[L(x)+(\Tr(x)+b)^s.\]
Permutation polynomials of this form have been studied in \cite{zeng2015permutation}, where $q$ is even and $s$ satisfies certain conditions.

In this paper, we study permutation rational functions of the form
\[L(x)+\frac c{\Tr(x)+b},\]
where $b\in\mathbb F_{q^n}\setminus\mathbb F_q$, $c\in\mathbb F_{q^n}^*$, and $L$ is a $q$-linear polynomial over $\mathbb F_{q^n}$. We shall characterize those permutation rational functions by factoring certain polynomials over $\overline{\mathbb F_q}$ in two indeterminates $X$ and $Y$, where $\overline{\mathbb F_q}$ denotes the algebraic closure of $\mathbb F_q$. When considering $\mathbb F_{q^n}(X,Y)/\mathbb F_q(X,Y)$, let $\Tr$ and $\N$ be defined as
\[\Tr(g)=\sum_{i=0}^{n-1}\sigma_q^i(g),\quad\text{and}\quad\N(g)=\prod_{i=0}^{n-1}\sigma_q^i(g)\]
for $g\in\mathbb F_{q^n}(X,Y)$, where $\sigma_q$ is the automorphism that maps each coefficient of polynomials to its $q$-th power and leaves $X$ and $Y$ fixed. The notation will be used throughout, as well as the following lemma.

\begin{lemma}[{\cite[Lemma 3.2]{aubry1996weil}}]\label{bound}
For an absolutely irreducible affine plane curve of degree $d$ over $\mathbb F_q$ with $d>1$, if
\[q-(d-1)(d-2)\sqrt q-2d+1>0,\]
then it has a rational point $(x_0,y_0)$ over $\mathbb F_q$ with $x_0\ne y_0$.
\end{lemma}

For the rational function $L(x)+\frac c{\Tr(x)+b}$, we divide the remainder of this paper into two parts. In Section \ref{s1}, we study the case that $L$ does not permute $\mathbb F_{q^n}$, and it turns out that such permutation rational functions tend not to exist. In Section \ref{s2}, permutation rational functions of the form $x+\frac c{\Tr(x)+b}$ are characterized, especially for small values of $n$.

\section{The Case That $L$ Does Not Permute $\mathbb F_{q^n}$}\label{s1}

Consider $\mathbb F_{q^n}$ as a vector space over $\mathbb F_q$. Recall that every linear map from $\mathbb F_{q^n}$ to $\mathbb F_q$ can be written as $\Tr(\beta x)$ for some $\beta\in\mathbb F_{q^n}$, and that for a basis $(\beta_1,\dots,\beta_n)$ of $\mathbb F_{q^n}/\mathbb F_q$, its dual basis $(\alpha_1,\dots,\alpha_n)$ satisfies
\[x=\alpha_1\Tr(\beta_1x)+\dots+\alpha_n\Tr(\beta_nx).\]
In other words, the map
\[x\mapsto(\Tr(\beta_1x),\dots,\Tr(\beta_nx))\]
is isomorphic from $\mathbb F_{q^n}$ to $\mathbb F_q^n$.

Let $r$ be the rank of $L$ as a linear endomorphism of $\mathbb F_{q^n}/\mathbb F_q$, and assume $r<n$. Let $(\alpha_1,\dots,\alpha_r)$ be a basis of the image of $L$ over $\mathbb F_q$, so that
\[L(x)=\sum_{i=1}^r\alpha_i\ell_i(x),\]
where $\ell_i$ is a map from $\mathbb F_{q^n}$ to $\mathbb F_q$, and
\[x\mapsto(\ell_1(x),\dots,\ell_r(x))\]
defines a map from $\mathbb F_{q^n}$ onto $\mathbb F_q^r$. In fact, $\ell_i$ is a linear map over $\mathbb F_q$, so
\[L(x)=\sum_{i=1}^r\alpha_i\Tr(\beta_ix)\]
for some $\beta_i\in\mathbb F_{q^n}$, where $\beta_1,\dots,\beta_r$ are linearly independent over $\mathbb F_q$.

Suppose $L(x)+\frac c{\Tr(x)+b}$ is a permutation rational function of $\mathbb F_{q^n}$. Then there do not exist $x_1,y_1\in\mathbb F_{q^n}$ with $x_1\ne y_1$ such that $\Tr(x_1)=\Tr(y_1)$ and $L(x_1)=L(y_1)$, which means the map
\[x\mapsto(\Tr(x),\Tr(\beta_1x),\dots,\Tr(\beta_rx))\]
is injective from $\mathbb F_{q^n}$ to $\mathbb F_q^{r+1}$. This implies that $r=n-1$, and the above map is also surjective. Assume $\frac{c(x_0-y_0)}{(x_0+b)(y_0+b)}$ belongs to the image of $L$ for some $x_0,y_0\in\mathbb F_q$ with $x_0\ne y_0$. Choose an element $y_1\in\mathbb F_{q^n}$ such that $\Tr(y_1)=y_0$, and $x_1\in\mathbb F_{q^n}$ determined by
\[\Tr(x_1)=x_0\quad\text{and}\quad L(x_1)=L(y_1)+\frac{c(x_0-y_0)}{(x_0+b)(y_0+b)}.\]
Clearly we have $x_1\ne y_1$ and
\[L(x_1)+\frac c{\Tr(x_1)+b}-L(y_1)-\frac c{\Tr(y_1)+b}=L(x_1-y_1)+\frac{c(y_0-x_0)}{(x_0+b)(y_0+b)}=0,\]
a contradiction, so the assumption fails to hold. On the other hand, there exists an element $\beta\in\mathbb F_{q^n}^*$ such that $\Tr(\beta\alpha_i)=0$ for $1\le i<n$, and thus $a\in\mathbb F_{q^n}$ is in the image of $L$ only if $\Tr(\beta a)=0$. The image of $L$ and the kernel of $\Tr(\beta x)$ in $\mathbb F_{q^n}$ both have dimension $n-1$ over $\mathbb F_q$, so they are actually equal. Consequently, we have
\[\Tr\left(\frac{\beta c}{(x_0+b)(y_0+b)}\right)\ne0\]
for all $x_0,y_0\in\mathbb F_q$ with $x_0\ne y_0$. Since $c\in\mathbb F_{q^n}^*$ is arbitrary, we may let $\beta=1$ and $L(x)=x^q-x$. However, such permutation rational functions do not exist in some cases. Before showing that, we present an auxiliary result.

\begin{lemma}\label{basis}
For $n=3$ and $b\in\mathbb F_{q^3}\setminus\mathbb F_q$, $(1,b^q+b,b^{q+1})$ is a basis of $\mathbb F_{q^3}/\mathbb F_q$.
\end{lemma}
\begin{proof}
In view of \cite[Corollary 2.38]{lidl1997}, the result follows from the fact that
\[\begin{split}\begin{vmatrix}1&b^q+b&b^{q+1}\\1&b^{q^2}+b^q&b^{q^2+q}\\1&b^{q^2}+b&b^{q^2+1}\end{vmatrix}&=\Tr\big(\big(b^{q^2}+b^q\big)b^{q^2+1}-b^{q^2+q}\big(b^{q^2}+b\big)\big)\\&=\N(b)\Tr\big(b^{q-1}-b^{q^2-1}\big)\\&\ne0.\end{split}\]
To see the inequality, note first that $b^{q-1}\ne b^{1-q}$. If $b^{q-1}\in\mathbb F_q$, then $(b^{q-1})^3=\N(b^{q-1})=1$, which means $q\equiv1\pmod3$ with
\[\Tr\big(b^{q-1}-b^{q^2-1}\big)=3(b^{q-1}-b^{1-q})\ne0;\]
if $b^{q-1}\notin\mathbb F_q$, then the minimal polynomial of $b^{q-1}$ over $\mathbb F_q$ is
\[x^3-\Tr(b^{q-1})x^2+\Tr\big((b^{q-1})^{q+1}\big)x-\N(b^{q-1})=x^3-\Tr(b^{q-1})x^2+\Tr\big(b^{q^2-1}\big)x-1,\]
whereas the assumption $\Tr\big(b^{q-1}-b^{q^2-1}\big)=0$ yields a root $1$ of it. 
\end{proof}

\begin{proposition}
For $n=2$ or $n=3$ and sufficiently large $q$, $L(x)+\frac c{\Tr(x)+b}$ is not a permutation rational function of $\mathbb F_{q^n}$ if $L$ does not permute $\mathbb F_{q^n}$.
\end{proposition}
\begin{proof}
It suffices to show that
\[\Tr\left(\frac c{(x_0+b)(y_0+b)}\right)=0\]
for some $x_0,y_0\in\mathbb F_q$ with $x_0\ne y_0$. If $n=2$, then
\[\Tr(c^q(x_0+b)(y_0+b))=\Tr(c)x_0y_0+\Tr(bc^q)(x_0+y_0)+\Tr(b^2c^q),\]
where one of $\Tr(c)$ and $\Tr(bc^q)$ is nonzero since $b\notin\mathbb F_q$ and $c\ne0$, and it is easy to see that for $q>3$,
\[\Tr\left(\frac c{(x_0+b)(y_0+b)}\right)=\frac{\Tr(c^q(x_0+b)(y_0+b))}{\N((x_0+b)(y_0+b))}=0\]
for some $x_0,y_0\in\mathbb F_q$ with $x_0\ne y_0$.

Consider the case $n=3$ with the polynomial
\[f(X,Y)=\Tr\big(c^{q^2}(X+b)(X+b^q)(Y+b)(Y+b^q)\big)\]
over $\mathbb F_q$. Note that the coefficient of $Y^2$ in $f(X,Y)$ over $\mathbb F_q(X)$ is $\Tr\big(c^{q^2}(X+b)(X+b^q)\big)$, which is nonzero because $\Tr(c)=\Tr\big((b^q+b)c^{q^2}\big)=\Tr\big(b^{q+1}c^{q^2}\big)=0$ only if $c=0$ by the above lemma, so $f(X,Y)$ has degree at least $2$. If $f(X,Y)$ is absolutely irreducible, then by Lemma \ref{bound},
\[\Tr\left(\frac c{(x_0+b)(y_0+b)}\right)=\frac{f(x_0,y_0)}{\N((x_0+b)(y_0+b))}=0\]
for some $x_0,y_0\in\mathbb F_q$ with $x_0\ne y_0$, when $q-6\sqrt q-7>0$.

Assume $f(X,Y)$ is reducible over $\overline{\mathbb F_q}$. If $g(X)$ divides $f(X,Y)$ for some monic polynomial $g$ of positive degree over $\overline{\mathbb F_q}$, then so does $g(Y)$ by symmetry; otherwise $f(X,Y)$ is reducible over $\overline{\mathbb F_q}(X)$. In either case, let $\xi\in\overline{\mathbb F_q}(X)$ be a root of $f(X,Y)$ as a polynomial over $\mathbb F_q(X)$; that is,
\[\sum_{i=0}^2c^{q^{i+2}}\big(X+b^{q^i}\big)\big(X+b^{q^{i+1}}\big)\big(\xi+b^{q^i}\big)\big(\xi+b^{q^{i+1}}\big)=0.\]
Let $v_i$ ($i\in\{0,1,2\}$) be the valuation of $\overline{\mathbb F_q}(X)$ corresponding to $X+b^{q^i}$. Clearly it can not occur that $v_0(\xi+b)=v_0(\xi+b^q)=v_0\big(\xi+b^{q^2}\big)=0$. If $v_0(\xi+b)\ne0$, then either $v_0(\xi+b)>v_0(\xi+b^q)=v_0\big(\xi+b^{q^2}\big)=0$ or $v_0(\xi+b)=v_0(\xi+b^q)=v_0\big(\xi+b^{q^2}\big)<0$, which is false again. Then we have $v_0(\xi+b)=0$ and $v_0\big(\xi+b^{q^j}\big)\ne0$ with either $j=1$ or $j=2$. Now the case that $g(Y)$ divides $f(X,Y)$ has been excluded, and by factoring $f(X,Y)$ one finds
\[\xi+b^{q^j}=\frac{\alpha_0X+\beta_0}{\alpha_1X+\beta_1}\]
for some $\alpha_0,\beta_0,\alpha_1,\beta_1\in\overline{\mathbb F_q}$, where $\alpha_1=1$ or $\alpha_1=0$ with $\beta_1=1$. 

If $v_0\big(\xi+b^{q^j}\big)<0$, then $\alpha_1=1$ and $\beta_1=b$. Now suppose $v_0\big(\xi+b^{q^j}\big)>0$. We have
\[\xi+b^{q^j}=\frac{\alpha_0X+\alpha_0b}{\alpha_1X+\beta_1},\]
with
\[\xi+b^{q^{j+1}}=\frac{\alpha_0X+\alpha_1\big(b^{q^{j+1}}-b^{q^j}\big)X+\alpha_0b+\beta_1\big(b^{q^{j+1}}-b^{q^j}\big)}{\alpha_1X+\beta_1},\]
and
\[\xi+b^{q^{j+2}}=\frac{\alpha_0X+\alpha_1\big(b^{q^{j+2}}-b^{q^j}\big)X+\alpha_0b+\beta_1\big(b^{q^{j+2}}-b^{q^j}\big)}{\alpha_1X+\beta_1}.\]
Investigating $v_1$ and $v_2$ likewise gives $v_1\big(\xi+b^{q^{j+1}}\big)=v_2\big(\xi+b^{q^{j+2}}\big)=1$, and hence,
\[\alpha_0b^q+\alpha_1\big(b^{q^{j+1}}-b^{q^j}\big)b^q=\alpha_0b+\beta_1\big(b^{q^{j+1}}-b^{q^j}\big),\]
\[\alpha_0b^{q^2}+\alpha_1\big(b^{q^{j+2}}-b^{q^j}\big)b^{q^2}=\alpha_0b+\beta_1\big(b^{q^{j+2}}-b^{q^j}\big).\]
If $\alpha_1=1$, then
\[\begin{pmatrix}b^q-b&b^{q^j}-b^{q^{j+1}}\\b^{q^2}-b&b^{q^j}-b^{q^{j+2}}\end{pmatrix}\begin{pmatrix}\alpha_0\\\beta_1\end{pmatrix}=\begin{pmatrix}\big(b^{q^j}-b^{q^{j+1}}\big)b^q\\\big(b^{q^j}-b^{q^{j+2}})b^{q^2}\end{pmatrix},\]
where
\[\begin{vmatrix}b^{q^j}-b^{q^{j+1}}&\big(b^{q^j}-b^{q^{j+1}}\big)b^q\\b^{q^j}-b^{q^{j+2}}&\big(b^{q^j}-b^{q^{j+2}})b^{q^2}\end{vmatrix}\ne0.\]
If $\alpha_1=0$, then
\[\begin{pmatrix}b^q-b&b^{q^j}-b^{q^{j+1}}\\b^{q^2}-b&b^{q^j}-b^{q^{j+2}}\end{pmatrix}\begin{pmatrix}\alpha_0\\1\end{pmatrix}=\begin{pmatrix}0\\0\end{pmatrix}.\]
Accordingly, if $\alpha_0$ and $\beta_1$ satisfy the linear equations, then they must belong to $\mathbb F_{q^3}$. Therefore, in all cases, either $\xi\in\mathbb F_q(X)$, indicating the existence of $x_0,y_0\in\mathbb F_q$ with $x_0\ne y_0$ such that $f(x_0,y_0)=0$ when $q>3$, or $\xi$, $\sigma_q(\xi)$ and $\sigma_q^2(\xi)$ are distinct roots of the quadratic polynomial $f(X,Y)$ over $\overline{\mathbb F_q}(X)$, a contradiction.
\end{proof}

\section{The Case That $L$ Permutes $\mathbb F_{q^n}$}\label{s2}

Suppose now that $L$ permutes $\mathbb F_{q^n}$, whose inverse permutation polynomial is denoted by $L^{-1}$. If $L(x)+\frac c{\Tr(x)+b}$ is a permutation rational function of $\mathbb F_{q^n}$, then substituting $L^{-1}(x)$ for $x$ we have
\[x+\frac c{\Tr(L^{-1}(x))+b}=x+\frac c{\Tr(\alpha x)+b}\]
for some $\alpha\in\mathbb F_{q^n}$. Hence, we only consider $x+\frac c{\Tr(x)+b}$ for $b\in\mathbb F_{q^n}\setminus\mathbb F_q$ and $c\in\mathbb F_{q^n}^*$, and give the following criterion.

\begin{lemma}\label{equiv}
The rational function
\[x+\frac c{\Tr(x)+b}\]
is a permutation rational function of $\mathbb F_{q^n}$, if and only if
\[x+\Tr\left(\frac c{x+b}\right)\]
is a permutation rational function of $\mathbb F_q$, if and only if
\[\Tr\left(\frac c{(x_0+b)(y_0+b)}\right)\ne1\]
for all $x_0,y_0\in\mathbb F_q$ with $x_0\ne y_0$.
\end{lemma}
\begin{proof}
For arbitrary $t_1\in\mathbb F_q$, we have $\Tr(x_1)=t_1$ for some $x_1\in\mathbb F_{q^n}$, and if
\[x_0+\frac c{\Tr(x_0)+b}=x_1\]
for some $x_0\in\mathbb F_{q^n}$, then
\[\Tr(x_0)+\Tr\left(\frac c{\Tr(x_0)+b}\right)=\Tr(x_1);\]
let $t_0=\Tr(x_0)$, so that
\[t_0+\Tr\left(\frac c{t_0+b}\right)=t_1.\]
Conversely, for arbitrary $x_1\in\mathbb F_{q^n}$, if the above equation holds for some $t_0\in\mathbb F_q$ and $t_1=\Tr(x_1)$, then let
\[x_0=x_1-\frac c{t_0+b},\]
so that $\Tr(x_0)=t_0$ and
\[x_0+\frac c{\Tr(x_0)+b}=x_0+\frac c{t_0+b}=x_1.\]
This proves the first equivalence.

Next, it is clear that
\[x+\Tr\left(\frac c{x+b}\right)\]
permutes $\mathbb F_q$ if and only if for all $x_0,y_0\in\mathbb F_q$ with $x_0\ne y_0$, one has
\[x_0+\Tr\left(\frac c{x_0+b}\right)\ne y_0+\Tr\left(\frac c{y_0+b}\right).\]
In fact,
\[\begin{split}&\mathrel{\phantom{=}}x_0+\Tr\left(\frac c{x_0+b}\right)-y_0-\Tr\left(\frac c{y_0+b}\right)\\&=x_0-y_0+\Tr\left(\frac{c(y_0-x_0)}{(x_0+b)(y_0+b)}\right)\\&=(x_0-y_0)\left(1-\Tr\left(\frac c{(x_0+b)(y_0+b)}\right)\right),\end{split}\]
and the result follows.
\end{proof}

This enables us to determine whether $x+\frac c{\Tr(x)+b}$ is a permutation rational function of $\mathbb F_{q^n}$. In the following, we first study the case $n=2$ and then the case $n=3$.

\begin{theorem}
For $n=2$, $x+\frac c{\Tr(x)+b}$ is a permutation rational function of $\mathbb F_{q^2}$ if and only if $c=(b^q-b)^{q+1}$.
\end{theorem}
\begin{proof}
Assume that the polynomial
\[f(X,Y)=\N((X+b)(Y+b))-\Tr(c^q(X+b)(Y+b))\]
is reducible over $\overline{\mathbb F_q}$. As a quadratic polynomial over $\overline{\mathbb F_q}[X]$, it has leading coefficient $\N(X+b)$; if it is not primitive, then it must be divisible by $X+b$ or $X+b^q$ in $\overline{\mathbb F_q}[X,Y]$, but this is not valid for $\Tr(c^q(X+b)(Y+b))$. Thus $f(X,Y)$ has a factor $g_1(X)Y+g_0(X)$ for some nonzero polynomials $g_0$ and $g_1$ over $\overline{\mathbb F_q}$ with $g_1$ monic. Denote by $v_\infty$ be the infinite valuation of $\overline{\mathbb F_q}(X)$ and let $\xi=-\frac{g_0(X)}{g_1(X)}\in\overline{\mathbb F_q}(X)$ so that
\[\N(X+b)(\xi+b)(\xi+b^q)=c^q(X+b)(\xi+b)+c(X+b^q)(\xi+b^q).\]
If $v_\infty(\xi+b)\le0$ and $v_\infty(\xi+b^q)\le0$, then
\[\begin{split}&\mathrel{\phantom{=}}v_\infty(\N(x+b)(\xi+b)(\xi+b^q))\\&=-2+v_\infty((\xi+b)(\xi+b^q))\\&<\min\{-1+v_\infty(\xi+b),-1+v_\infty(\xi+b^q)\}\\&\le v_\infty(c^q(x+b)(\xi+b)+c(x+b^q)(\xi+b^q)),\end{split}\]
which gives rise to a contradiction. Then let $v_\infty(\xi+b)>0$ without loss of generality, and consequently $v_\infty(\xi+b^q)=0$. Moreover,
\[v_\infty(\N(X+b)(\xi+b)(\xi+b^q))=v_\infty(\xi+b)-2,\]
and
\[v_\infty(c^q(X+b)(\xi+b)+c(X+b^q)(\xi+b^q))=-1,\]
so
\[v_\infty\left(\frac{bg_1(x)-g_0(x)}{g_1(x)}\right)=v_\infty(\xi+b)=1.\]
This happens only if $\deg(g_0)=\deg(g_1)=1$ and $g_0$ has leading coefficient $b$. We may write
\[g_0(X)=bX+\beta_0\quad\text{and}\quad g_1(X)=X+\beta_1\]
for some $\beta_0,\beta_1\in\overline{\mathbb F_q}$. Since $f(X,Y)\in\mathbb F_q[X,Y]$ is fixed by $\sigma_q$, it has another factor $XY+b^qX+\beta_1^qY+\beta_0^q$ in $\overline{\mathbb F_q}[X,Y]$, relatively prime to $XY+bX+\beta_1Y+\beta_0$. Comparing the leading coefficients, we get
\[f(X,Y)=(XY+bX+\beta_1Y+\beta_0)(XY+b^qX+\beta_1^qY+\beta_0^q)\]
with $\beta_0,\beta_1\in\mathbb F_{q^2}$. Since $f(X,Y)$ is symmetric, we have either $\beta_1=b$ or $\beta_1=b^q$ and $\beta_0=\beta_0^q$. Equating the coefficients of $XY$, $X$ and the constant terms from
\[\begin{split}&\mathrel{\phantom{=}}\N((X+b)(Y+b))-\Tr(c^q(X+b)(Y+b))\\&=(XY+bX+\beta_1Y+\beta_0)(XY+b^qX+\beta_1^qY+\beta_0^q)\end{split}\]
yields
\[\Tr(b)^2-\Tr(c)=\Tr(b\beta_1^q+\beta_0),\]
\[\Tr(b)\N(b)-\Tr(bc^q)=\Tr(b\beta_0^q),\]
and
\[\N(b)^2-\Tr(b^2c^q)=\N(\beta_0).\]
If $\beta_1=b$, then
\[\Tr(b^2-c-\beta_0)=\Tr(b^q(b^2-c-\beta_0))=0,\]
and thus
\[\N(b)^2-\Tr(b^2c^q)=\N(b^2-c)=\N(b)^2-\Tr(b^2c^q)+\N(c),\]
which implies $c=0$. It remains to consider the case $\beta_1=b^q$ and $\beta_0=\beta_0^q$. Now we have
\[2\N(b)-\Tr(c)-2\beta_0=\Tr(b)\N(b)-\Tr(bc^q)-\Tr(b)\beta_0=\N(b)^2-\Tr(b^2c^q)-\beta_0^2=0,\]
and then
\[\begin{split}\Tr(b^2c^q)&=\Tr((\Tr(b)b-\N(b))c^q)\\&=\Tr(b)\Tr(bc^q)-\N(b)\Tr(c)\\&=\Tr(b)^2(\N(b)-\beta_0)-2\N(b)(\N(b)-\beta_0)\\&=\Tr(b^2)(\N(b)-\beta_0),\end{split}\]
where $\N(b)-\beta_0\ne0$, for otherwise $\Tr(c)=\Tr(bc^q)=0$. Consequently,
\[\N(b)+\beta_0=\Tr(b^2),\]
and
\[\Tr(c)=2\N(b)-2\Tr(b^2)+2\N(b)=2\N(b^q-b),\]
while
\[\Tr(bc^q)=\Tr(b)(2\N(b)-\Tr(b^2))=\Tr(b)\N(b^q-b).\]
Such an element $c\in\mathbb F_{q^2}^*$ is uniquely determined by $b$, and we have $c=\N(b^q-b)$ satisfying the above equations.

We have proved that $f(X,Y)$ is absolutely irreducible if $c\ne\N(b^q-b)$. In this case, there exist $x_0,y_0\in\mathbb F_q$ with $x_0\ne y_0$ such that
\[1-\Tr\left(\frac c{(x_0+b)(y_0+b)}\right)=\frac{f(x_0,y_0)}{\N((x_0+b)(y_0+b))}=0\]
when $q-6\sqrt q-7>0$, by Lemma \ref{bound}. The rest is verified by exhaustive search. If $c=\N(b^q-b)$ for arbitrary $b\in\mathbb F_{q^2}\setminus\mathbb F_q$, then
\[f(X,Y)=\N(XY+bX+b^qY+\Tr(b^2)-\N(b)),\]
but $x_0y_0+bx_0+b^qy_0+\Tr(b^2)-\N(b)=0$ for $x_0,y_0\in\mathbb F_q$ only if
\[x_0y_0+\Tr(b)x_0+\Tr(b^2)-\N(b)=b^q(x_0-y_0),\]
where the right side lies outside of $\mathbb F_q$ if $x_0\ne y_0$.
\end{proof}

\begin{remark}
Let $\alpha$ be an element in $\mathbb F_{q^2}$ with $\alpha^{q-1}=-1$. Substituting $\alpha^{-1}x$ for $x$ in $x+\frac c{x^q+x+b}$, we get
\[\alpha^{-1}x+\frac c{\alpha^{-q}x^q+\alpha^{-1}x+b}=\alpha^{-1}\left(x+\frac{\alpha^{q+1}c}{x^q-x+\alpha^qb}\right).\]
This extends the result of \cite{hou2020type}, i.e., whether
\[x+\frac1{x^q-x+\alpha^qb}\]
is a permutation rational function of $\mathbb F_{q^2}$ (note that $b\notin\mathbb F_q$ if and only if $\Tr(\alpha^qb)\ne0$).
\end{remark}

\begin{theorem}
For $n=3$, $x+\frac c{\Tr(x)+b}$ is a permutation rational function of $\mathbb F_{q^3}$ if $c=-(b^q-b)^{q^2+1}$. The converse holds for sufficiently large $q$.
\end{theorem}
\begin{proof}
We claim that if $c\ne-(b^q-b)^{q^2+1}$, then the polynomial
\[f(X,Y)=\N((X+b)(Y+b))-\Tr\big(c^{q^2}(X+b)(X+b^q)(Y+b)(Y+b^q)\big)\]
over $\mathbb F_q$ is absolutely irreducible; otherwise,
\[f(X,Y)=\N((X+b)(Y+b)-c).\]
If $f(X,Y)$ is absolutely irreducible, then by Lemma \ref{bound}, there exist $x_0,y_0\in\mathbb F_q$ such that
\[1-\Tr\left(\frac c{(x_0+b)(y_0+b)}\right)=\frac{f(x_0,y_0)}{\N((x_0+b)(y_0+b))}=0\]
with $x_0\ne y_0$, when $q-20\sqrt q-11>0$. Next, let $f(X,Y)=\N((X+b)(Y+b)-c)$ with $c=-(b^q-b)^{q^2+1}$, and assume $f(x_0,y_0)=0$ for some $x_0,y_0\in\mathbb F_q$ with $x_0\ne y_0$. It follows that $(x_0+b)(y_0+b)=c$, in which case,
\[\Tr(1)=\Tr\left(\frac c{(x_0+b)(y_0+b)}\right)=1-\frac{f(x_0,y_0)}{\N((x_0+b)(y_0+b))}=1,\]
and then $q$ is even. Note that
\[(x_0+b)(y_0+b)-c=x_0y_0+b(x_0+y_0)+b^2+(b^q-b)^{q^2+1}=x_0y_0+b(x_0+y_0)+\Tr(b^{q+1}),\]
but $b(x_0+y_0)\notin\mathbb F_q$, a contradiction.

In what follows, we prove the claim. Assume $f(X,Y)$ is reducible over $\overline{\mathbb F_q}$. By a similar argument, it is a primitive polynomial over $\overline{\mathbb F_q}[X]$ and has a factor $g_1(X)Y+g_0(X)$ for some nonzero polynomials $g_0$ and $g_1$ over $\overline{\mathbb F_q}$ with $g_1$ monic. For the infinite valuation $v_\infty$ of $\overline{\mathbb F_q}(X)$ and $\xi=-\frac{g_0(X)}{g_1(X)}\in\overline{\mathbb F_q}(X)$, which satisfies
\[\N(X+b)(\xi+b)(\xi+b^q)\big(\xi+b^{q^2}\big)=\sum_{i=0}^2c^{q^{i+2}}\big(X+b^{q^i}\big)\big(X+b^{q^{i+1}}\big)\big(\xi+b^{q^i}\big)\big(\xi+b^{q^{i+1}}\big),\]
we exclude the case that $v_\infty(\xi+b)\le0$, $v_\infty(\xi+b^q)\le0$ and $v_\infty\big(\xi+b^{q^2}\big)\le0$, and let $v_\infty(\xi+b)>v_\infty(\xi+b^q)=v_\infty\big(\xi+b^{q^2}\big)=0$ as in the proof of the last theorem. Furthermore,
\[v_\infty\big(\N(X+b)(\xi+b)(\xi+b^q)\big(\xi+b^{q^2}\big)\big)=v_\infty(\xi+b)-3,\]
\[v_\infty\left(\sum_{i=0}^2c^{q^{i+2}}\big(X+b^{q^i}\big)\big(X+b^{q^{i+1}}\big)\big(\xi+b^{q^i}\big)\big(\xi+b^{q^{i+1}}\big)\right)=-2,\]
and thus,
\[v_\infty\left(\frac{bg_1(X)-g_0(X)}{g_1(X)}\right)=v_\infty(\xi+b)=1.\]
This means $g_0$ has leading coefficient $b$ with $\deg(g_0)=\deg(g_1)$. It is clear that $f(X,Y)$ is fixed by the automorphism $\sigma_q$ of $\overline{\mathbb F_q}[X,Y]$, while $g_0(X)\notin\mathbb F_q[X]$, so $g_1(X)Y+g_0(X)$, $\sigma_q(g_1(X)Y+g_0(X))$ and $\sigma_q^2(g_1(X)Y+g_0(X))$ are distinct irreducible polynomials in $\overline{\mathbb F_q}[X,Y]$ dividing $f(X,Y)$. Moreover, they are relatively prime in $\overline{\mathbb F_q}[X,Y]$ since $g_1$ is monic. Therefore, $f(X,Y)$ is divisible by their product. In fact, noticing that $f(X,Y)$ is a cubic polynomial in one indeterminate and comparing the leading coefficients, we obtain
\[f(X,Y)=(g_1(X)Y+g_0(X))\sigma_q(g_1(X)Y+g_0(X))\sigma_q^2(g_1(X)Y+g_0(X)),\]
where $\deg(g_0)=\deg(g_1)=1$. 

We now show that $g_1(X)Y+g_0(X)$ is symmetric. Assume the contrary so that
\[g_1(X)Y+g_0(X)\ne g_1(Y)X+g_0(Y).\]
By the same argument, $g_1(X)Y+g_0(X)$ and $g_1(Y)X+g_0(Y)$ are relatively prime, and thus $f(X,Y)$ as a symmetric polynomial in $\overline{\mathbb F_q}[X,Y]$ is divisible by $(g_1(X)Y+g_0(X))(g_1(Y)X+g_0(Y))$. We may write
\[f(X,Y)=(g_1(X)Y+g_0(X))(g_1(Y)X+g_0(Y))h(X,Y),\]
where $h(X,Y)$ is a linear polynomial over $\overline{\mathbb F_q}[X]$ as well as a symmetric polynomial in $\overline{\mathbb F_q}[X,Y]$, but neither $\sigma_q(g_1(X)Y+g_0(X))$ nor $\sigma_q^2(g_1(X)Y+g_0(X))$ is symmetric. Eventually,
\[f(X,Y)=\N((X+b)(Y+b)+\gamma),\]
for some $\gamma\in\mathbb F_{q^3}$. Expanding $\N((X+b)(Y+b)+\gamma)$ leads to
\[\Tr\big((c+\gamma)^{q^2}(X+b)(X+b^q)(Y+b)(Y+b^q)\big)+\Tr\big(\gamma^{q^2+q}(X+b)(Y+b)\big)+\N(\gamma)=0.\]
Looking at the coefficients of $X^2Y^2$, $X^2Y$ and $X^2$, we assert that
\[\Tr(c+\gamma)=\Tr\big((b^q+b)(c+\gamma)^{q^2}\big)=\Tr\big(b^{q+1}(c+\gamma)^{q^2}\big)=0,\]
which implies $c+\gamma=0$ by Lemma \ref{basis}. This allows one to write
\[\Tr\big(c^{q^2+q}(x+b)(y+b)\big)-\N(c)=0,\]
and, by computing the coefficients,
\[\Tr(c^{q+1})=\Tr\big(bc^{q^2+q}\big)=\Tr\big(b^2c^{q^2+q}\big)-\N(c)=0.\]
Divide it by $\N(c)$ to get
\[\Tr(c^{-1})=\Tr(bc^{-1})=\Tr(b^2c^{-1})-1=0,\]
where $c^{-1}$ is uniquely determined by $b$, since $(1,b,b^2)$ is a basis of $\mathbb F_{q^3}/\mathbb F_q$. If $c=-(b^q-b)^{q^2+1}$, then
\[c^{q+1}=(b^q-b)^{(q^2+1)(q+1)}=\N(b^q-b)(b^q-b),\]
which implies
\[\Tr(c^{q+1})=0,\]
\[\Tr\big(bc^{q^2+q}\big)=\N(b^q-b)\Tr\big(b^{q^2+1}-b^{q+1}\big)=0,\]
and
\[\Tr\big(b^2c^{q^2+q}\big)-\N(c)=\N(b^q-b)\Tr\big(b^{q^2+2}-b^{q+2}\big)+\N(b^q-b)^2=0,\]
where $\N(b^q-b)=\Tr\big(b^{q^2+2}-b^{q+2}\big)$ as easily seen. In this case, $f(X,Y)$ is factored as desired. The proof of the claim is then complete.
\end{proof}

\begin{remark}
Suppose $q$ is odd, $b\in\mathbb F_{q^3}\setminus\mathbb F_q$ and $c=-(b^q-b)^{q^2+1}$. It has been shown that
\[\Tr\left(\frac c{x_0y_0+b(x_0+y_0)+b^2}\right)\ne1\]
for all $x_0,y_0\in\mathbb F_q$. In fact, one can prove a slightly stronger result that
\[\Tr\left(\frac c{u+bv+b^2}\right)\ne1\]
for all $u,v\in\mathbb F_q$. Let $z=u+bv+b^2$ and assume $\Tr(cz^{-1})=1$. We have just seen that
\[\Tr(c^{-1})=\Tr(bc^{-1})=\Tr(b^2c^{-1})-1=0,\]
and then $\Tr(zc^{-1})=1$. If $cz^{-1}\in\mathbb F_q$, then
\[9=(3cz^{-1})(3zc^{-1})=\Tr(cz^{-1})\Tr(zc^{-1})=1,\]
and $q$ is even. If $cz^{-1}\notin\mathbb F_q$, then its minimal polynomial over $\mathbb F_q$ is
\[x^3-\Tr(cz^{-1})x^2+\N(cz^{-1})\Tr(zc^{-1})x-\N(cz^{-1})=x^3-x^2+\N(cz^{-1})x-\N(cz^{-1}),\]
which, however, is reducible with a root $1$.
\end{remark}

The above sufficient conditions can be extended as follows.

\begin{corollary}
Let $\Tr_{i,j}$ denote the trace map from $\mathbb F_{q^i}$ to $\mathbb F_{q^j}$. Then $x+\frac c{\Tr(x)+b}$ is a permutation rational function of $\mathbb F_{q^n}$ if
\begin{itemize}
\item $2$ divides $n$, $b\in\mathbb F_{q^2}\setminus\mathbb F_q$ and $\Tr_{n,2}(c)=(b^q-b)^{q+1}$, or
\item $3$ divides $n$, $b\in\mathbb F_{q^3}\setminus\mathbb F_q$ and $\Tr_{n,3}(c)=-(b^q-b)^{q^2+1}$.
\end{itemize}
\end{corollary}
\begin{proof}
By Lemma \ref{equiv}, the rational function permutes $\mathbb F_{q^n}$ if and only if
\[x+\Tr\left(\frac c{x+b}\right)\]
permutes $\mathbb F_q$. If $2$ divides $n$ and $b\in\mathbb F_{q^2}\setminus\mathbb F_q$, then
\[x+\Tr\left(\frac c{x+b}\right)=x+\Tr_{2,1}\circ\Tr_{n,2}\left(\frac c{x+b}\right)=x+\Tr_{2,1}\left(\frac{\Tr_{n,2}(c)}{x+b}\right),\]
as a rational function over $\mathbb F_q$. The same holds for the case that $3$ divides $n$ and the result follows immediately from the above on quadratic and cubic extensions.
\end{proof}

\section{Conclusions}

We have studied permutation rational functions of $\mathbb F_{q^n}$ in the form
\[L(x)+\frac c{\Tr(x)+b}\]
for $b\in\mathbb F_{q^n}\setminus\mathbb F_q$ and $c\in\mathbb F_{q^n}^*$. Specifically, we have investigated the rational function
\[x+\frac c{\Tr(x)+b}\]
for $n=2$ or $n=3$. One may consider those permutation rational functions for greater $n$. However, factorization of the corresponding polynomials in $\mathbb F_{q^n}[X,Y]$ is much more complicated, so it could be difficult to apply the aforementioned method. For future work, it would be interesting to find a more general treatment for this class of permutation rational functions and study their properties.

\bibliographystyle{abbrv}
\bibliography{references}

\end{document}